\documentclass[12pt]{article}
\usepackage{amsmath}
\usepackage{amsthm}
\usepackage{amsfonts}
\usepackage[latin5]{inputenc}
\numberwithin{equation}{section}
\newcounter{thmcounter}
\newcounter{Remarkcounter}

\numberwithin{thmcounter}{section}

\newtheorem{Prop}[thmcounter]{Proposition}
\newtheorem{Corol}[thmcounter]{Corollary}
\newtheorem{theorem}[thmcounter]{Theorem}
\newtheorem{Lemma}[thmcounter]{Lemma}
\theoremstyle{remark}
\newtheorem{Remark}[Remarkcounter]{Remark}

\begin{document}
\title{On the quasi-minimal surfaces in the 4-dimensional de Sitter space with 1-type Gauss map}


\author{Nurettin Cenk Turgay \footnote{
Address: Istanbul Technical University, Faculty of Science and Letters,
Department of  Mathematics, 34469 Maslak, Istanbul, Turkey} \footnote{e-mail:turgayn@itu.edu.tr}}

\maketitle
\begin{abstract}                                                            
In this paper, we study the Gauss map of the surfaces in the de Sitter space-time $\mathbb S^4_1(1)$. First, we prove that a space-like surface lying in the  de Sitter space-time has pointwise 1-type Gauss map if and only if it has parallel mean curvature vector. Then, we obtain the complete classification of the quasi-minimal  surfaces with  1-type Gauss map. 


\end{abstract}

\section{Introduction}
The notion of finite type mappings have been extensively studied by several geometers after it was introduced by B. Y. Chen in late 1970's. Many results in this topic have been published so far, \cite{ChenRapor}. Even now, there are several open problems on this subject which are currently being dealt with.

Let $\mathbb E^m_s$ denote the semi-Euclidean space with dimension $m$ and index $s$ whose metric tensor is given by 
$$\tilde g=\langle\ ,\ \rangle=-\sum\limits_{i=1}^sx_i^2+\sum\limits_{j=s+1}^mx_j^2$$ 
and $M$ be an oriented $n$-dimensional semi-Euclidean submanifold of $\mathbb E^m_s$. Consider  a smooth mapping $\phi$ defined on $M$ into another semi-Euclidean space  $\mathbb E^N_S$. $\phi$ is said to be $k$-type if it can be expressed as a sum of
\begin{equation}\label{FiniteTypeGaussDef}
 \phi=\phi_0+\phi_1+\phi_2+\hdots+\phi_k,
\end{equation}
where $\phi_0$ is a constant vector and $\phi_i$ is a non-constant eigenvector of $\Delta$ corresponding to eigenvalue $\lambda_i$ for $i=1,2,\hdots,k$ with  $\lambda_1<\lambda_2<\hdots< \lambda_k$ and $\Delta$  is the  Laplace operator of $M$ with respect to the induced metric of $M$, \cite{Chen-Morvan-Nore}. Note that if the position vector of $M$ is $k$-type, then $M$ is said to be of $k$-type, \cite{ChenKitap,ChenMakale1986}.
                          
In particular, if the mapping $\phi$ is the Gauss map of $M$, then $M$ is said to have $k$-type  Gauss map, \cite{Chen-Piccinni}.  From these definitions, one can immediately see that an oriented submanifold $M$ of  the semi-Euclidean space $\mathbb E^m_s$ has 1-type Gauss map if and only if its Gauss map $\nu$ satisfies 
\begin{equation}\label{Glbl1TypeDefinition}
 \Delta \nu  = \lambda (\nu+C)
\end{equation}
for a constant $\lambda$ and  a constant vector $C$. On the other hand, if $\nu$ satisfies 
\begin{equation}\label{PW1TypeDefinition}
 \Delta \nu =f(\nu +C)
\end{equation}
for a smooth non-constant function $f$ and a constant vector $C$, then $M$ is said to have proper pointwise 1-type Gauss map. 

A surface $M$ in a 4-dimensional Lorentzian manifold is said to be quasi-minimal if its mean curvature vector is light-like on every point of $M$. In this work, we study the quasi-minimal surfaces of the de Sitter space-time $\mathbb S^4_1(1)$ in terms of the type of their Gauss map. In the Section 2, after we describe the general notion that we use, we give a brief summary of the basic facts and definitions. In the section 3, we obtain the complete classification of  quasi-minimal surfaces of de Sitter space-time with 1-type Gauss map. We also give a characterization of these type of surfaces with proper pointwise 1-type Gauss map. 

\section{Prelimineries}
In this section, we give the basic definitions and facts, \cite{ONeillKitap}. We also mention about the notations which   we will use in this paper, which are along the lines used in \cite{NCTGenRelGrav}.

We put 
\begin{eqnarray} 
\mathbb S^{m-1}_s(r^2,c_0)&=&\{x\in\mathbb E^m_s: \langle x-c_0, x-c_0 \rangle=r^{-2}\},\notag
\\  
\mathbb H^{m-1}_{s-1}(-r^2,c_0)&=&\{x\in\mathbb E^m_s: \langle x-c_0, x-c_0\rangle=-r^{-2}\},\notag
\end{eqnarray}
where $\langle\ ,\ \rangle$ is the indefinite inner product of $\mathbb E^m_s$. In general relativity, $\mathbb E^4_1$, $\mathbb S^{4}_1 (r^2)=\mathbb S^{4}_1 (r^2,0)$ and  $\mathbb H^{4}_1 (r^2)=\mathbb H^{4}_1 (r^2,0)$
are known as the Minkowski, de Sitter and anti-de Sitter space-times, respectively, \cite{ChenVeken2009Houston}.

Now, consider a space-like surface $M$ of  de Sitter space-time  $\mathbb S^{4}_1 (1)$.  We put $\nabla$ and $\widetilde\nabla$ for the Levi Civita connections of  $M$ and $\mathbb S^{4}_1 (1)$ respectively. Then, the Gauss and Weingarten formulas of $M$ become
\begin{eqnarray}
\label{MEtomGauss} \widetilde\nabla_X Y&=& \nabla_X Y + h(X,Y),\\
\label{MEtomWeingarten} \widetilde\nabla_X \xi&=& -A_\xi X+D_X \xi
\end{eqnarray}
for any tangent vector field $X,\ Y$ and normal vector field $\xi$,  where  $h$ and $D$   are the second fundemental form and the normal connection of $M$ in  $\mathbb S^{4}_1 (1)$, respectively and $A$ is the shape operator of $M$. We denote the curvature tensor associated with the connections $\nabla$ and $D$ by $R$ and $R^D$, respectively. 

The   Codazzi  equation is given by
\begin{equation}
\label{MinkCodazzi} D_X h(Y,Z)-h(\nabla_X Y,Z)-h(Y,\nabla_X Z)=D_Y h(X,Z)-h(\nabla_Y X,Z)-h(X,\nabla_Y Z).
\end{equation}
The mean curvature vector $H$ of  $M$ in  $\mathbb S^4_1(1)$ is defined by $H=\frac 12 \mathrm{tr}h$. If $H$ is light-like on $M$, then $M$ is said to be a quasi-minimal surface in $\mathbb S^4_1(1)$.

\subsection{Gauss map}\label{SubSectMinkGaussMap}
Let  $\Lambda^{n}(\mathbb E^m_s)$ and $\tilde G(n, m)$ denote the space of  $n$-vectors on $\mathbb E^m_s$ and the Grassmannian manifold consisting of all oriented $n$-planes through the origin of $\mathbb E^m_s$, respectively.  Note that $\tilde G(n, m)$  is canonically imbedded in $\Lambda^{n}(\mathbb E^m_s)$ which is an $N$ dimensional vector field, where $N= {m\choose {n}}$. A non-degenerate inner product on $\Lambda^{n}(\mathbb E^m_s)$ is defined by
$$\langle X_1\wedge X_2\wedge\cdots\wedge X_{n}, Y_1\wedge Y_2\wedge\cdots\wedge Y_{n}\rangle= \det(\langle X_i,Y_j\rangle),$$
where $X_i,\;Y_i\in\mathbb E^m_{s},\; i=1,2,\hdots,n$ and $\langle X_i,Y_j\rangle$ denotes the semi-Euclidean inner product of the vectors $X_i$ and $Y_j$. We will denote the inner product space $\Big(\Lambda^{n}(\mathbb E^m_s),\langle,\rangle\Big)$ by $\Lambda^{m,n}_S$, where $S$ is its index. There exists a one-to-one, onto and linear isometry from $\Lambda^{m,n}_S$ into $\mathbb E^N_S$, because their dimension and index are equal(see \cite[p. 52]{ONeillKitap}). Hence, we have $\tilde G(n, m)\subset\Lambda^{m,n}_S\cong\mathbb E^N_S$.  

Let $M$ be an $n$-dimensional, oriented space-like submanifold of the semi-Euclidean space  $\mathbb E^m_s$. Consider a local orthonormal base field $\{e_1, e_2,\hdots,e_{n}\}$ of the tangent bundle of $M$. Then, the Laplace operator of $M$ with respect to the induced metric is 
\begin{equation}\label{SemiEuclSpacSubmflDelta}
\Delta=\sum\limits^n_{i=1}(-e_ie_i+\nabla_{e_i}e_i).
\end{equation}
The smooth mapping
\begin{equation}\label{MinkGaussTasvTanim}
\begin{array}{rcl}\nu:M&\rightarrow&\tilde G(n, m)\subset R^{N-1}_S (1)\subset \mathbb E^N_S\cong\Lambda^{m,n}_S\\
p&\mapsto&\nu(p)=(e_{1}\wedge e_{2}\wedge\hdots\wedge e_n)(p)\end{array}
\end{equation}
is called the (tangent) Gauss map of $M$  which assigns a point $p$ in $M$   to the representation of the oriented 
$n$-plane through  the origin of $\mathbb E^m_s$ and parallel to the tangent  space of $M$ at $p$.

$M$ is said to have 1-type Gauss map if \eqref{Glbl1TypeDefinition} is satisfied for a constant $\lambda$ and a constant vector $C\in \mathbb E^6_3$. Moreover,  $M$ is said to have pointwise 1-type Gauss map if \eqref{PW1TypeDefinition} is satisfied for a smooth  function $f$ and a constant vector $C\in \mathbb E^6_3$, \cite{Chen-Piccinni,KKKM}. A pointwise 1-type Gauss map is called proper if \eqref{PW1TypeDefinition} is satisfied for a non-constant  function $f$.

\section{Surfaces with pointwise 1-type Gauss map}
Let $M$ be a space-like surface in $\mathbb E^m_s$ and $\nu$ its Gauss map. Consider an orthonormal frame field $\{e_1,e_2;e_3,e_4,\hdots,e_m\}$. From \cite[Lemma 3.2]{KKKM}, one can obtain that $\nu$ satisfies 
\begin{equation}\label{Pre121Mink4GaussLaplMargTrapped}
\Delta\nu=\|\hat h\|^2\nu+\sum\limits_{3\leq\alpha\leq\beta\leq m}\varepsilon_\alpha\varepsilon_\beta\langle R^{\hat D}(e_1,e_2)e_\alpha,e_\beta\rangle e_\alpha\wedge e_\beta-2\hat D_{e_1}\hat H\wedge e_2-2e_1\wedge \hat D_{e_2}\hat H,
\end{equation}
where  $\hat D$, $\hat h$ and $\hat H$ denote normal connection, second fundemental form and mean curvature vector of $M$ in $\mathbb E^m_s$, respectively, $R^{\hat D}$ is the curvature tensor associated with $\hat D$ and $\|\hat h\|^2$ is the squared norm of $\hat h$. 

\subsection{Space-like surfaces in $\mathbb S^4_1(1)$}
Now, consider a surface $M$ in $\mathbb S^4_1(1)\subset\mathbb E^5_1$  and let $x$ be its position vector in $\mathbb E^5_1$. 
We want to note that the following equalities is satisfied for any vector fields $\xi,\ \eta$ normal to $M$ and tangent to $\mathbb S^4_1(1)$
\begin{align}\nonumber
\begin{split}
&h(e_i,e_j)=\hat h(e_i,e_j)+\delta_{ij} x,\\
& R^{\hat D}(e_1,e_2;\xi,x)=0, \quad R^{\hat D}(e_1,e_2;\xi,\eta)=R^{D}(e_1,e_2;\xi,\eta),\\
&\hat D_{e_i}\hat H=D_{e_i}H,
\end{split}
\end {align} 
where  $\hat D$, $\hat h$ and $\hat H$ denote normal connection, second fundemental form and mean curvature vector of $M$ in $\mathbb E^5_1$, respectively, and $R^{\hat D}$ is the curvature tensor associated with $\hat D$.
By taking into account these equations, we obtain from \eqref{Pre121Mink4GaussLaplMargTrapped} that
\begin{align}\label{Mink4GaussLaplMargTrppd}
\begin{split}
\Delta\nu=&\left(4-2K+\langle H,H\rangle\right)\nu -2R^D(e_1,e_2;e_3,e_4)e_3\wedge e_4\\&-2D_{e_1}H\wedge e_2-2e_1\wedge D_{e_2}H.
\end{split}
\end{align}

Now, we assume that $M$ has pointwise 1-type Gauss map, i.e., \eqref{PW1TypeDefinition} is satisfied for a smooth  function $f$ and a constant vector $C\in \mathbb E^6_3$. From \eqref{PW1TypeDefinition} and \eqref{Mink4GaussLaplMargTrppd} we obtain
\begin{align}\nonumber
\begin{split}
f(\nu+C)=&\left(4-2K+\langle H,H\rangle\right)\nu -2R^D(e_1,e_2;e_3,e_4)e_3\wedge e_4\\&-2D_{e_1}H\wedge e_2-2e_1\wedge D_{e_2}H.
\end{split}
\end{align}
from which we get 
\begin{equation}\label{CxeAlr}
\langle C,x\wedge e_A\rangle=0,\quad A=1,2,3,4.
\end{equation}
from which we obtain 
\begin{equation}\label{CxeAlrei}
e_i(\langle C,x\wedge e_A\rangle)=0,\quad i=1,2. 
\end{equation}
As $C$ is a constant vector, \eqref{CxeAlrei} implies
\begin{equation}\label{CxeAlreiaaa}
\langle C,e_i(x\wedge e_A)\rangle=0.
\end{equation}
By using \eqref{MEtomGauss} and \eqref{MEtomWeingarten}, we obtain 
\begin{align}\label{CxeAlreiaaaaa}
e_i(x\wedge e_A) =e_i\wedge e_A+x\wedge \zeta,
\end{align}
where $\zeta$ is a vector field tangent to $S^4_1(1).$ From \eqref{CxeAlr}-\eqref{CxeAlreiaaaaa} we get
\begin{equation}\label{NihaiSonuc}
\langle C,e_i\wedge e_A\rangle=0,\quad A=1,2,3,4,\ i=1,2.
\end{equation}
Thus, we obtain that $C$ is of the form of $C=C_{34}e_3\wedge e_4$. By a further calculation, we have $C_{34}=0$ which yields $C=0$. Thus, we have

\begin{Prop}\label{PropPW1TYPE}
 Let $M$ be  a surface in $\mathbb S^4_1(1)$. If  $M$ has pointwise 1-type Gauss map, then \eqref{PW1TypeDefinition} is satisfied for $f=4-2K+\langle H,H\rangle$ and $C=0.$
\end{Prop}
Now we want to give the following corollaries of this proposition
\begin{Corol}
 Let $M$ be  a space-like surface in $\mathbb S^4_1(1)$ with non-zero mean curvature vector $H$. Then,  $M$ has pointwise 1-type Gauss map if and only if $H$ is parallel. 
\end{Corol}
\begin{proof}
Now, we assume that $M$ has pointwise 1-type Gauss map and $H\neq0$. Then, Proposition \ref{PropPW1TYPE} implies that \eqref{PW1TypeDefinition}  is satisfied for $C=0$ and  $f=4-2K+\langle H,H\rangle$. From  \eqref{PW1TypeDefinition}  and \eqref{Mink4GaussLaplMargTrppd} we have 
\begin{equation}\label{MaxNonmaxAraDenk01}
\left(4-2K+\langle H,H\rangle\right)\nu -2R^D(e_1,e_2;e_3,e_4)e_3\wedge e_4-2D_{e_1}H\wedge e_2-2e_1\wedge D_{e_2}H=f\nu
\end{equation}
which imply $D_{e_i}H=0$, i.e., $H$ is parallel. 

Conversely, if $H$ is parallel, then the normal bundle of $M$ is flat. Thus \eqref{Mink4GaussLaplMargTrppd} implies $\Delta\nu=\left(4-2K+\langle H,H\rangle\right)\nu$. Hence, $M$ has pointwise 1-type Gauss map. 
\end{proof}

\begin{Corol}
 Let $M$ be  a space-like surface in $\mathbb S^4_1(1)$ with zero mean curvature vector. Then,  $M$ has pointwise 1-type Gauss map if and only if $M$ has flat normal bundle. 
\end{Corol}
\begin{proof}
Now, we assume $H=0$. Then, because of Proposition \ref{PropPW1TYPE} and \eqref{Mink4GaussLaplMargTrppd}, $M$ has pointwise 1-type Gauss map if and only if 
\begin{equation}\label{MaxNonmaxAraDenk02}
\left(4-2K+\langle H,H\rangle\right)\nu -2R^D(e_1,e_2;e_3,e_4)e_3\wedge e_4=f\nu
\end{equation}
is satisfied for a smooth function $f$. 
\end{proof}

\subsection{Gauss map of quasi-minimal surfaces in $\mathbb S^4_1(1)$}
The author obtained the following results in \cite{NCTGenRelGrav}.
\begin{Prop}\label{Glbl1TypePROP}\cite{NCTGenRelGrav}
Let $M$ be a marginally trapped surface in the de Sitter  space-time. If $M$ has 1-type Gauss map, then $\Delta\nu=4\nu$ or $\Delta\nu=2\nu$.
\end{Prop}
\begin{Corol}\cite{NCTGenRelGrav}
There is no marginally trapped surface in the de Sitter  space-time with harmonic Gauss map.
\end{Corol}
In this subsection, we will give the complete classification of quasi-minimal surfaces in $\mathbb S^4_1(1)$ with 1-type Gauss map.

 In \cite{ChenVeken2009Houston}, the classification of quasi-minimal surfaces in $\mathbb S^4_1(1)$ with parallel mean curvature vector is given. It is obtained that a quasi-minimal surface $M$ has parallel mean curvature vector in $\mathbb S^4_1(1)$ if and only if it is congruent to an open part of the following eight type of surfaces:
\begin{enumerate}
\item[(i)] A surface given by 
\begin{equation}\label{S411ParallelHcase1}
x(u,v)=(1,\sin u,\cos u\cos v,\cos u\sin v,1);
\end{equation}
\item[(ii)] A surface given by 
\begin{equation}\label{S411ParallelHcase2}
\displaystyle x(u,v)=\frac 12(2u^2-1,2u^2-2,2u,\sin 2v,\cos 2v);
\end{equation}
\item[(iii)] A surface given by 
\begin{equation}\label{S411ParallelHcase3}
x(u,v)=\left(\frac b{cd},\frac {\cos cu}{c},\frac {\sin cu}{c},\frac {\cos dv}{d},\frac {\sin dv}{d}\right),
\end{equation}
 where $c=\sqrt{2-b}$ and $d=\sqrt{2+b}$ with $|b|<2$;
\item[(iv)]  A surface given by 
\begin{equation}\label{S411ParallelHcase4}
x(u,v)=\left(\frac {\cosh cu}{c},\frac {\sinh cu}{c},\frac {\cos dv}{d},\frac {\sin dv}{d},\frac b{cd}\right),
\end{equation}
 where $c=\sqrt{b-2}$ and $d=\sqrt{b+2}$ with $|b|>2$;
\item[(v)] A surface of curvature one  with constant light-like mean curvature vector, lying in $K_a=\{(t,x_2,x_3,x_4,t+a)| t,x_2,x_3,x_4\in\mathbb R\}$;
\item[(vi)] A surface of curvature one lying in $\mathcal {LC}_1=\{(y,1)|\langle y,y \rangle=0,y\in\mathbb E^4_1\}$;
\item[(vii)] A surface lying in $\mathbb S^4_1(1)\cap \mathbb S^4(c_0,r^2),$ where $c_0\neq0$ and $r>0$;
\item[(viii)] A surface lying in $\mathbb S^4_1(1)\cap \mathbb H^4(c_0,-r^2),$ where $c_0\neq0$ and $r>0$.
\end{enumerate}
\begin{Remark}\label{REmmark01}
By a simple calculation, one can see that the surfaces given by \eqref{S411ParallelHcase1}-\eqref{S411ParallelHcase4} have constant Gaussian curvature. Therefore, by taking into account Proposition \ref{PropPW1TYPE}, it is easy to obtain that all of the surfaces given in case (i)-(vi) has 1-type Gauss map. 
\end{Remark}

In the next lemmas, we will show that the surfaces given in case (vii) and (viii) have non-constant Gaussian curvature.
\begin{Lemma}\label{ALEmmma01}
A quasi-minimal surface in $\mathbb S^4_1(1)$ lying in $\mathbb S^4_1(1)\cap \mathbb S^4(c_0,r^2)$ has  non-constant Gaussian curvature and parallel mean curvature vector in $\mathbb S^4_1(1)$,  where $c_0\neq0$ and $r>0$.
\end{Lemma}
\begin{proof}
Let $M$ be a quasi-minimal surface in $\mathbb S^4_1(1)$ lying in $\mathbb S^4_1(1)\cap \mathbb S^4(c_0,r^2)$. 
Then, we have $\langle x,x\rangle=1$ and $\langle x-c_0,x-c_0\rangle=r^{-2}$. These equations imply 
\begin{subequations}\label{ALEmmma01Denk01}
\begin{eqnarray}
\label{ALEmmma01Denk01a}\langle x,X\rangle&=&0\\
\label{ALEmmma01Denk01b}\langle x-c_0,X\rangle&=&0\\
\label{ALEmmma01Denk01c}\langle c_0,X\rangle&=&0
\end{eqnarray}
\end{subequations}
for all vector fields $X$ tangent to $M$ and 
\begin{equation}
\label{ALEmmma01Denk02}\langle x,c_0\rangle=c
\end{equation}
for a constant $c$.

Now, we define a vector field $\xi$ on $M$ as $\xi=\langle c_0,x\rangle x-c_0$. Note that $\xi$ is normal to $M$, tangent to $\mathbb S^4_1(1)$ because of \eqref{ALEmmma01Denk01b} and \eqref{ALEmmma01Denk01c}. Moreover, \eqref{ALEmmma01Denk01a} and \eqref{ALEmmma01Denk02} imply $\langle\xi,\xi\rangle=a$ for a constant $a$. From \eqref{ALEmmma01Denk02} we have
$\widetilde\nabla_{X}\xi=\langle c_0,x\rangle X$. Thus, $\xi$ is parallel and the shape operator along $\xi$ is proportional to identity operator by a constant.

Now, we will show that the  Gaussian curvature $K$ of $M$ is non-constant. We assume that $K$ is constant and consider the orthonormal base field $\{e_3,e_4\}$ of the normal bundle of $M$ such that $e_4$ is proportional to $\xi$. As $H$ is parallel, light-like and $K$ is constant, we may choose a base field $\{e_1,e_2\}$ of the tangent bundle of $M$ such that $A_3=\mathrm{diag}(b-c,b+c)$ and  $A_4=bI$ for some constants $b$ and $c$. From the Codazzi equation \eqref{MinkCodazzi}, we have
$$\omega_{12}(e_i)(h(e_1,e_1)-h(e_2,e_2))=0,\quad i=1,2$$
 which implies $K=0$ which yields a contradiction. Hence, the proof is completed.

\end{proof}
Similarly, we have
\begin{Lemma}\label{ALEmmma02}
A quasi-minimal surface lying in $\mathbb S^4_1(1)\cap \mathbb H^4(c_0,-r^2)$ has  non-constant Gaussian curvature and parallel mean curvature vector,  where $c_0\neq0$ and $r>0$.
\end{Lemma}
By combaining Proposition \ref{PropPW1TYPE}, Remark \ref{REmmark01}, Lemma \ref{ALEmmma01} and  Lemma \ref{ALEmmma02} we obtain the following results.
\begin{theorem}
Let $M$ be a quasi-minimal surface lying in $\mathbb S^4_1(1)$. Then $M$ has 1-type Gauss map if and only if it is congruent to a surface congruent to either one of the surfaces given by \eqref{S411ParallelHcase1}-\eqref{S411ParallelHcase4} or the following two type of surfaces:
\begin{enumerate}
\item[(i)] A surface of curvature one  with constant light-like mean curvature vector, lying in $K_a=\{(t,x_2,x_3,x_4,t+a)| t,x_2,x_3,x_4\in\mathbb R\}$;
\item[(ii)] A surface of curvature one lying in $\mathcal {LC}_1=\{(y,1)|\langle y,y \rangle=0,y\in\mathbb E^4_1\}$.
\end{enumerate}
\end{theorem}
\begin{theorem}
Let $M$ be a quasi-minimal surface lying in $\mathbb S^4_1(1)$. Then $M$ has proper pointwise 1-type Gauss map if and only if it is congruent to a surface congruent to one of the  the following two type of surfaces:
\begin{enumerate}
\item[(i)] A surface lying in $\mathbb S^4_1(1)\cap \mathbb S^4(c_0,r^2),$ where $c_0\neq0$ and $r>0$;
\item[(ii)] A surface lying in $\mathbb S^4_1(1)\cap \mathbb H^4(c_0,-r^2),$ where $c_0\neq0$ and $r>0$.
\end{enumerate}
\end{theorem}

\section*{Acknowledgements}
  The author is supported by the Scientific Research Agency of Istanbul Technical University.

\end{document}